\documentclass[12pt,reqno]{amsart}
\usepackage{geometry}                
\usepackage{amsmath,amssymb}
\usepackage{graphicx}
\usepackage{amssymb,latexsym, amsmath}
\usepackage{amsfonts,amsthm}
\usepackage{amscd}
\usepackage{mathrsfs}
\usepackage{hyperref}
\usepackage{eucal}
\usepackage{epstopdf}
\usepackage{amsgen}
\usepackage{xspace}
\usepackage{verbatim}
\usepackage{stmaryrd}
\usepackage{enumitem}
\newlist{steps}{enumerate}{1}
\setlist[steps, 1]{label = Step \arabic*:}

\newcommand{\gd}{\Delta}
\newcommand{\ggd}{\delta}
\newcommand{\tu}{\Tilde{u}}
\newcommand{\tv}{\Tilde{v}}
\newcommand{\tw}{\Tilde{w}}

\newcommand{\tg}{\Tilde{g}}

\newcommand{\inpt}[1]{\langle #1 \rangle}

\newcommand{\gw}{\Omega}

\newcommand{\ga}{\gamma}

\newcommand{\ms}{\mathscr}

\newcommand{\nb}{\nabla}
\newcommand{\vp}{\varphi}
\newcommand{\ve}{\varepsilon}
\newcommand{\pdr}{\partial}

\newcommand{\tup}{\textup}
\newcommand{\csg}{\{ S(t)\}_{t\geq0}}

\newcommand{\beq}{\begin{equation}}
\newcommand{\eeq}{\end{equation}}
\newcommand{\bea}{\begin{align}}
\newcommand{\eea}{\end{align}}
\newcommand{\bthm}{\begin{theorem}}
\newcommand{\ethm}{\end{theorem}}
\newcommand{\bpr}{\begin{proof}}
\newcommand{\epr}{\end{proof}}
\newcommand{\bcl}{\begin{corollary}}
\newcommand{\ecl}{\end{corollary}}
\newcommand{\bpn}{\begin{proposition}}
\newcommand{\epn}{\end{proposition}}
\newcommand{\bre}{\begin{remark}}
\newcommand{\ere}{\end{remark}}
\newcommand{\bdf}{\begin{definition}}
\newcommand{\edf}{\end{definition}}
\newcommand{\bss}{\begin{align*}}
\newcommand{\ess}{\end{align*}}

\newcommand{\bl}{\label}

\newtheorem{theorem}{Theorem}[section]
\newtheorem{corollary}[theorem]{Corollary}

\newtheorem{lemma}[theorem]{Lemma}
\newtheorem{proposition}[theorem]{Proposition}

\theoremstyle{definition}
\newtheorem{definition}[theorem]{Definition}
\theoremstyle{remark}
\newtheorem{remark}{Remark}

\numberwithin{equation}{section}

\begin{document}

\title[Hindmarsh-Rose Equations]{Exponential Attractor for Hindmarsh-Rose Equations in Neurodynamics}

\author[C. Phan]{Chi Phan}
\address{Department of Mathematics and Statistics, University of South Florida, Tampa, FL 33620, USA}
\email{chi@mail.usf.edu}
\thanks{}

\author[Y. You]{Yuncheng You}
\address{Department of Mathematics and Statistics, University of South Florida, Tampa, FL 33620, USA}
\email{you@mail.usf.edu}
\thanks{}


\date{August 14, 2019}



\begin{abstract} 
The existence of an exponential attractor for the diffusive Hindmarsh-Rose equations on a three-dimensional bounded domain originated in the study of neurodynamics is proved through uniform estimates together with a new theorem on the squeezing property of an abstract reaction-diffusion equation also proved in this paper. The results infer that the global attractor whose existence has been established in \cite{PYS} for the Hindmarsh-Rose semiflow has a finite fractal dimension. 
\end{abstract}

\maketitle

\section{\textbf{Introduction}}

The Hindmarsh-Rose equations for neuronal spiking-bursting of the intracellular membrane potential observed in experiments was originally proposed in \cite{HR1, HR2}. This mathematical model composed of three coupled ordinary differential equations has been studied through numerical simulations and mathematical analysis in recent years, cf. \cite{HR1, HR2, IG, MFL, SPH, Su} and the references therein. 

We shall study in this paper the global dynamics in terms of the existence of an exponential attractor for the diffusive Hindmarsh-Rose equations:
\begin{align}
    \frac{\pdr u}{\pdr t} & = d_1 \gd u +  \vp (u) + v - w + J, \, \bl{ueq} \\
    \frac{\pdr v}{\pdr t} & = d_2 \gd v + \psi (u) - v, \bl{veq} \\
    \frac{\pdr w}{\pdr t} & = d_3 \gd w + q (u - c) - rw, \bl{weq}
\end{align}
for $t > 0,\; x \in \gw \subset \mathbb{R}^{n}$ ($n \leq 3$), where $\gw$ is a bounded domain with locally Lipschitz continuous boundary, $J $ is a constant and the nonlinear terms 
\beq \bl{pp}
	\vp (u) = au^2 - bu^3, \quad \text{and} \quad \psi (u) = \alpha - \beta u^2.
\eeq 
In this system \eqref{ueq}-\eqref{weq}, the variable $u(t,x)$ refers to the membrane electric potential of a neuronal cell, the variable $v(t, x)$ represents the transport rate of the ions of sodium and potassium through the fast ion channels and is called the spiking variable, while the variables $w(t, x)$ represents the transport rate across the neuronal cell membrane through slow channels of calcium and other ions and is called the bursting variable. 

All the involved parameters are positive constants except $c \,(= u_R) \in \mathbb{R}$, which is a reference value of the membrane potential of a neuron cell. In the original model of ODE \cite{Su}, a set of the typical parameters are
\begin{gather*}
	J = 3.281, \;\; r = 0.0021, \;\; S = 4.0, \; \; q = rS,  \;\; c = -1.6,  \\[3pt]
	 \vp (s) = 3.0 s^2 - s^3, \;\; \psi (s) = 1.0 - 5.0 s^2.
\end{gather*}
We impose the homogeneous Neumann boundary conditions for the three components,
\begin{equation} \label{nbc}
    \frac{\pdr u}{\pdr \nu} (t, x) = 0, \; \; \frac{\pdr v}{\pdr \nu} (t, x)= 0, \; \; \frac{\pdr w}{\pdr \nu} (t, x)= 0,\quad  t > 0,  \; x \in \partial \gw ,
\end{equation}
and the initial conditions to be specified are denoted by
\begin{equation} \bl{inc}
    u(0, x) = u_0 (x), \quad v(0, x) = v_0 (x), \quad w(0, x) = w_0 (x), \quad x \in \gw.
\end{equation}

\subsection{\textbf{The Hindmarsh-Rose Model in ODE}}

In 1982-1984, J.L. Hindmarsh and R.M. Rose developed a mathematical model to describe neuronal activity and dynamics:
\begin{equation} \label{HR}
	\begin{split}
    \frac{du}{dt} & = au^2 - bu^3 + v - w + J,  \\
    \frac{dv}{dt} & = \alpha - \beta u^2  - v,  \\
    \frac{dw}{dt} & =  q (u - u_R) - rw.
    \end{split}
\end{equation}

This neuron model was motivated by the discovery of neuronal cells in the pond snail \emph{Lymnaea} which generated a burst after being depolarized by a short current pulse. This model characterizes the phenomena of synaptic bursting and especially chaotic bursting in a three-dimensional $(u, v, w)$ space. 

Neuronal signals are short electrical pulses called spikes or action potential. Neurons often exhibit bursts of alternating phases of rapid firing spikes and then quiescence. Bursting constitutes a mechanism to modulate and set the pace for brain functionalities and to communicate signals with the neighbor neurons. 

Bursting behaviors and patterns occur in a variety of excitable cells and bio-systems such as pituitary melanotropic gland, thalamic neurons, respiratory pacemaker neurons, and insulin-secreting pancreatic $\beta$-cells, cf. \cite{BRS, CK,CS, HR2}. Neurons communicate and coordinate actions through synapses or diffusive coupling called gap junction in neuroscience. Synaptic coupling of neurons has to reach certain threshold for release of quantal vesicles and synchronization \cite{DJ, Ru, SC}. 

The mathematical analysis mainly using bifurcations together with numerical simulations of several models in ODEs on bursting behavior has been studied by many authors, cf. \cite{BB, ET, Iz, MFL, Ri, SPH, SR, Tr, WS, Su}. 

The chaotic coupling exhibited in the simulations and analysis of this Hindmarsh-Rose model in ordinary differential equations shows more rapid synchronization and more effective regularization of neurons due to \emph{lower threshold} than the synaptic coupling \cite{Tr}. It was rigorously proved in \cite{SPH, Su} that chaotic bursting solutions can be quickly synchronized and regularized when the coupling strength is large enough to topologically change the bifurcation diagram based on this Hindmarsh-Rose model, but the dynamics of chaotic bursting is highly complex.

It is known that Hodgkin-Huxley equations \cite{HH} (1952) provided a four-dimensional model for the dynamics of membrane potential taking into account of the sodium, potassium as well as leak ions current. It is a highly nonlinear system if without simplification assumptions. FitzHugh-Nagumo equations \cite{FH} (1961-1962) provided a two-dimensional model for an excitable neuron with the membrane potential and the current variable. This two-dimensional ODE model admits an exquisite phase plane analysis showing spikes excited by supra-threshold input pulses and sustained periodic spiking with refractory period, but due to the 2D nature FitzHugh-Nagumo equations exclude any chaotic solutions and chaotic dynamics so that no chaotic bursting can be generated. 

The Hindmarsh-Rose equations \eqref{HR} generate a significant mechanism for rapid firing and chaotic busting in neurodynamics. This model reflects possible lower down the neuron firing threshold and allows for spikes with varying interspike-interval. Therefore, this 3D model is a suitable choice for the investigation of both the regular bursting and the chaotic bursting when the parameters vary. The study of dynamical properties of the Hindmarsh-Rose equations exposes to a wide range of applications in neuroscience. 

The rest of Section 1 presents the formulation of the system \eqref{ueq}-\eqref{inc} and provides the relevant concepts and the recent results in \cite{PYS} on the existence of global attractor for this diffusive Hindmarsh-Rise equations as well as the existing theory on exponential attractors. In Section 2, we shall prove a general theorem of the squeezing property for the abstract reaction-diffusion equation on a higher dimensional bounded domain. In Section 3, the main result on the existence of exponential attractor is established for the semiflow generated by the diffusive Hindmarsh-Rose equations.

\subsection{\textbf{Formulation and Preliminaries}}

Neuron is a specialized biological cell in the brain and the central nervous system. In general, neurons are composed of the central cell body containing the nucleus and intracellular organelles, the dendrites, the axon, and the terminals. The dendrites are the short branches near the nucleus receiving incoming signals of voltage pulse and the axon is a long branch to propagate outgoing signals.

Neurons are immersed in aqueous chemical solutions consisting of different ions electrically charged. The cell membrane is the conductor along which the voltage signals travel. As pointed out in \cite{EI}, neuron is a distributed dynamical system. 

From physical and mathematical viewpoint, it is reasonable and useful to consider the Hindmarsh-Rose model in partial differential equations with the spatial variables $x$ involved, at least in $\mathbb{R}^1$.  Here in the abstract extent, we shall study the diffusive Hindmarsh-Rose equations \eqref{ueq}-\eqref{weq} on a bounded domain $\gw$ of the space $\mathbb{R}^3$ and focus on the global asymptotic dynamics of the solutions.

We start with formulation of the aforementioned initial-boundary value problem of \eqref{ueq}--\eqref{inc} into an abstract evolutionary equation. Define the Hilbert space $H = [L^2 (\gw)]^3 = L^2 (\gw, \mathbb{R}^3)$ and the Sobolev space $E =  [H^{1}(\gw)]^3 = H^1 (\gw, \mathbb{R}^3)$. The norm and inner-product of $H$ or $L^2 (\gw)$ will be denoted by $\| \, \cdot \, \|$ and $\inpt{\,\cdot , \cdot\,}$, respectively. The norm of $E$ will be denoted by $\| \, \cdot \, \|_E$. The norm of $L^p (\gw)$ or $L^p (\gw, \mathbb{R}^3)$ will be dented by $\| \cdot \|_{L^p}$ if $p \neq 2$. We use $| \, \cdot \, |$ to denote a vector norm in a Euclidean space.

The initial-boundary value problem \eqref{ueq}--\eqref{inc} is formulated as an initial value problem of the evolutionary equation:
\begin{equation} \label{pb}
 	\begin{split}
   	& \frac{\partial g}{\partial t} = A g + f(g), \quad t > 0, \\[2pt]
    	g &\, (0) = g_0 = (u_0, v_0, w_0) \in H.
	\end{split}
\end{equation}
Here the nonpositive self-adjoint operator
\begin{equation} \label{opA}
        A =
        \begin{pmatrix}
            d_1 \gd  & 0   & 0 \\[3pt]
            0 & d_2 \gd  & 0 \\[3pt]
            0 & 0 & d_3 \gd
        \end{pmatrix}
        : D(A) \rightarrow H,
\end{equation}
where $D(A) = \{g \in H^2(\gw, \mathbb{R}^3): \pdr g /\pdr \nu = 0 \}$ is the generator of an analytic $C_0$-semigroup $\{e^{At}\}_{t \geq 0}$ on the Hilbert space $H$ due to the Lumer-Phillips theorem  \cite{SY}. By the fact that $H^{1}(\gw) \hookrightarrow L^6(\gw)$ is a continuous imbedding for space dimension $n \leq 3$ and by the H\"{o}lder inequality, there is a constant $C_0 > 0$ such that 
$$
    \| \vp (u)  \| \leq C_0 \| u \|_{L^6}^3 \quad \tup{and} \quad \|\psi (u) \| \leq C_0 \| u \|_{L^4}^2 \quad \textup{for} \; u \in L^6 (\gw).
$$
Therefore, the nonlinear mapping 
\begin{equation} \label{opf}
    f(u,v, w) =
        \begin{pmatrix}
             \vp (u) + v - w + J \\[4pt]
            \psi (u) - v,  \\[4pt]
	     q (u - c) - rw
        \end{pmatrix}
        : E \longrightarrow H
\end{equation}
is a locally Lipschitz continuous mapping. We can simply write column vectors $g(t)$ as $(u(t, \cdot), v(t, \cdot ), w(t, \cdot))$ and write $g_0 = (u_0, v_0, w_0)$. Consider the weak solution of this initial value problem \eqref{pb}, cf. \cite[Section XV.3]{CV}, defined below. 
\begin{definition} \label{D:wksn}
	A function $g(t, x), (t, x) \in [0, \tau] \times \gw$, is called a \emph{weak solution} to the initial value problem \eqref{pb}, if the following conditions are satisfied:
	
	\textup{(i)} $\frac{d}{dt} (g, \zeta) = (Ag, \zeta) + (f(g), \zeta)$ is satisfied for a.e. $t \in [0, \tau]$ and for any $\zeta \in E$;
	
	\textup{(ii)} $g(t, \cdot) \in L^2 (0, \tau; E) \cap C_w ([0, \tau]; H)$ such that $g(0) = g_0$.
	
\noindent
Here $(\cdot , \cdot)$ stands for the dual product of $E^*$ and $E$, and $C_w$ stands for the weakly continuous functions valued in $H$. Moreover, a function $g(t, x), (t, x) \in [0, \tau] \times \gw$, is a \emph{strong solution} of this initial value problem \eqref{pb} if it is a weak solution and satisfies the condition of regularity in \eqref{soln} below on a time interval $[0, \tau]$ and if the evolutionary equation \eqref{pb} is satisfied in the space $H$ for almost every $t \in (0, \tau)$. 
\end{definition}

In \cite{PYS}, the two authors and J. Su proved the following result on the existence of global solutions to the initial value problem \eqref{pb}.

\begin{theorem} \label{Lm2}
For any given initial data $g_0 = (u_0, v_0, w_0) \in H$, there exists a unique global weak solution $g(t, g_0) ) = (u(t), v(t), w(t)), \, t \in [0, \infty)$, of the initial value problem \eqref{pb} for the diffusive Hindmarsh-Rose equations \eqref{ueq}-\eqref{weq}, which continuously depends on the initial data and satisfies 
\begin{equation} \label{soln}
    g \in C([0, \infty); H) \cap C^1 ((0, \infty); H) \cap L_{loc}^2 ([0, \infty); E).
\end{equation}
All the weak solution becomes a strong solution on the interval $(0, \infty)$. The time-parametrized mapping $\{S(t) g_0 = g(t, g_0), t \geq 0\}$ is called the Hindmarsh-Rose semiflow.
\end{theorem}

\subsection{\textbf{Global Attractor and Exponential Attractor}}

We refer to \cite{CLR, CV, Milani, Rb, SY, Tm} for the basic concepts and results in the theory of infinite dimensional dynamical systems, including the few listed here for clarity.

\begin{definition} \label{Dabsb}
Let $\{S(t)\}_{t \geq 0}$ be a semiflow on a Banach space $\ms{X}$. A bounded set $B_0$ of $\ms{X}$ is called an absorbing set for this semiflow, if for any given bounded subset $B \subset \ms{X}$ there is a finite time $T_0 \geq 0$ depending on $B$, such that $S(t)B \subset B_0$ for all $t \geq T_0$.
\end{definition}

\begin{definition}[Global Attractor] \label{Dgla}
A set $\mathscr{A}$ in a Banach space $\ms{X}$ is called a global attractor for a semiflow $\csg$ on $\ms{X}$, if the following two properties are satisfied:

(i) $\mathscr{A}$ is a nonempty, compact, and invariant set in the space $\ms{X}$. 

(ii) $\mathscr{A}$ attracts any given bounded set $B \subset \ms{X}$ in the sense 
$$
	\text{dist}_{\ms{X}} (S(t)B, \mathscr{A}) = \sup_{x \in B} \inf_{y \in \mathscr{A}} \| S(t)x - y \|_{\ms{X}} \to 0, \;\;  \text{as} \; \; t \to \infty.
$$
\end{definition}

Global attractor characterizes qualitatively the longtime, asymptotic, and global dynamics of all the solution trajectories of a PDE system. As specified in \cite{CV, Rb, SY, Tm} as well as in \cite{Y08, Y10}, global attractor is a depository (usually fractal finite-dimensional) of all the permanent regimes including chaotic structures of an infinite-dimensional dynamical system. Global dynamic patterns are also important in neural field and neural network theories \cite{Co, ET}. For the autocatalytic reaction-diffusion systems \cite{Y08, Y10, Y12, Y14}, it is proved that global attractors exist. 

Recently in \cite{PYS}, the two authors and J. Su proved the following theorems on the absorbing property of the Hindmarsh-Rose semiflow and the existence of global attractor for the diffusive Hindmarsh-Rose equations \eqref{pb}.

\begin{theorem} \label{AbE}
	For any given bounded set $B \subset H$, there exists a finite time $T_B > 0$ such that for any initial state $g_0 = (u_0, v_0, w_0) \in B$, the weak solution $g(t) = S(t)g_0 = (u(t), v(t), w(t))$ of the initial value problem \eqref{pb} uniquely exists for $t \in [0, \infty)$ and satisfies 
\beq \label{ac}
	\| (u(t), v(t), w(t))\|_E \leq Q, \quad \text {for} \;\; t \geq T_B,
\eeq
where $Q > 0$ is a constant independent of any bounded set $B$ in $H$, and the finite $T_B > 0$ only depends on the bounded set $B$.
\end{theorem}

We shall call the time-parametrized mapping $\{S(t)\}_{t \geq 0}$ the Hindmarsh-Rose semiflow, which is a dynamical system on the space $H$.

\begin{theorem}[\textbf{Global Attractor for Diffusive Hindmarsh-Rose Equations}]  \label{MTh}
	For any positive parameters $d_1, d_2, d_3, a, b, \alpha, \beta, q, r, J$ and $c \in \mathbb{R}$, there exists a global attractor $\mathscr{A}$ in the phase space $H = L^2 (\gw, \mathbb{R}^3)$ for the Hindmarsh-Rose semiflow $\{S(t)\}_{t \geq 0}$ generated by the weak solutions of the diffusive Hindmarsh-Rose equations \eqref{pb}. Moreover, the global attractor $\mathscr{A}$ is an $(H, E)$-global attractor.
\end{theorem}

Global attractor for an infinite-dimensional dynamical systems generated by evolutionary PDE may exhibit slow rates and complicated behavior in attraction of solution trajectories. The notion of exponential attractor was introduced in \cite{Eden}.  

\begin{definition}[Exponential Attractor] \label{EA}
Suppose that $X$ is a Banach space and $\{S(t)\}_{t \geq 0}$ is a semiflow on $X$. A subset $\mathscr{E} \subset X$ is called an exponential attractor for this semiflow if the following three conditions are satisfied:

1) $\mathscr{E}$ is a compact in $X$ with a finite fractal dimension.

2) $\mathscr{E}$ is positively invariant with respect to the semiflow $\{S(t)\}_{t \geq 0}$ in the sense 
$$
	S(t) \mathscr{E} \subset \mathscr{E} \quad  \text{for all}\;\; t \geq 0.
$$

3) $\mathscr{E}$ attracts all the solution trajectories exponentially with a uniform rate $\sigma > 0$ in the sense that for any given bounded set $B \subset X$ there is a constant $C(B) > 0$ and
$$
	\text{dist}_X (S(t)B, \mathscr{E}) \leq C(B) e^{- \sigma t}, \quad  t \geq 0.
$$
\end{definition}

If there exists an exponential attractor $\mathscr{E}$ (may not be unique) as well as a global attractor $\mathscr{A}$ for a semiflow in a Banach space $X$, then it is always true that 
$$
	\mathscr{A} \subset \mathscr{E}.
$$
Consequently, the global attractor must have a finite fractal dimension as a subset of the exponential attractor. 

There are two approaches in terms of  sufficient conditions for construction of an exponential attractor. The first approach is the squeezing property which was introduced in the book \cite{Eden} and expounded in \cite{Milani}. The second approach is the compact smoothing property introduced by Efendiev-Miranville-Zelik \cite {EMZ, EYY}. Conceptually, the two properties are essentially equivalent when the phase space is a Hilbert space. From the application viewpoint, the squeezing property fits more to the semilinear reaction-diffusion equations. The second approach has been exploited in proving the existence of exponential attractors for quasilinear reaction-diffusion systems \cite{Yagi}.

The following definition of squeezing property \cite{KL, Milani} for a mapping means that either the mapping (which can be s snapshot of a semiflow at any time $t^*$) is a contraction or that higher modes are dominated by lower modes. 

\begin{definition}[Sqeezing Property] \label{SQP}
	Let $H$ be a Hilbert space and $\{S(t)\}_{t\geq 0}$ be a semiflow on $H$ whose norm is $\|\cdot \|$. Let $S = S(t^*)$ for some fixed $t^* \in (0, \infty)$. 
If there is a positively invariant set $Z \subset H$ with respect to this semiflow and there is a constant $ 0 < \delta < 1$ and an orthogonal projection $P$ from $H$ onto a finite-dimensional subspace of $PH \subset H$ , such that either
	
$$
	\|Su - Sv\| \leq \delta \|u - v \|, \quad \text{for any} \; u, v \in Z,
$$
or
$$
	\|(I - P) (Su - Sv)\| \leq \|P (Su - Sv)\|,  \quad \text{for any} \; u, v \in Z,
$$
then we say that the mapping $S$ has the squeezing property and the affiliated semiflow $\{S(t)\}_{t\geq 0}$ has the squeezing property on the set $Z$. 
\end{definition}

\begin{definition}[Fractal Dimension]
     The fractional dimension of a bounded subset $\mathscr{M}$ in a Banach space is defined by
$$
	\text{dim}_f \, \mathscr{M} = \limsup_{\ve \to 0^+} \frac{\log N_{\ve} [\mathscr{M}]}{\log (1/\ve)}
$$
where $N_\ve [\mathscr{M}]$ is the infimum number of open balls with the radius $\ve$ for a covering of the set $\mathscr{M}$. 
\end{definition}
                        
The following theorem states sufficient conditions for the existence of an exponential attractor with respect to a semiflow in a Hilbert space. Its proof is seen in \cite[Theorems 4.4 and 4.5]{Milani}.

\begin{theorem} \label{ExpAr}
Let $\{S(t)\}_{t \geq 0}$ be a semiflow on a Hilbert space $H$ with the following conditions satisfied\textup{:}

\textup{1)} The squeezing property is satisfied for $S = S(t^*)$ at some $t^* > 0$ on a nonempty, 

\hspace{10pt} compact, positively invariant, and absorbing set $M \subset H$.

\textup{2)} For all $t \in [0, t^*]$, the mapping $S(t): M \to M$ is Lipschitz continuous and the 

\hspace{10pt} Lipschitz constant $K(t): [0, t^*] \to (0, \infty)$ is a bounded function.

\textup{3)} For any $g \in M$, the mapping $S(\cdot) g: [0, t^*] \to M$ is Lipschitz continuous and 

\hspace{10pt} the Lipschitz constant $L(g): M \to (0, \infty)$ is a bounded function.

\noindent Then there exists an exponential attractor $\mathscr{E}$ in the space $H$ for this semiflow. Moreover, for any $\theta \in (0, 1)$, the fractal dimension of the exponential attractor $\mathscr{E}$ has the estimate
\beq \label{dimF}
	\dim_F (\mathscr{E}) \leq N \max \left\{1, \, \frac{\log (2\sqrt{2}L/\theta + 1)}{- \log \theta}\right\}
\eeq
where $N$ is the rank of the spectral projection associated with the squeezing property of the mapping $S(t^*)$ and $L$ is the Lipschitz constant of the mapping $S(t^*)$ on the positively invariant absorbing set $M$. 
\end{theorem}

\section{\textbf{Squeezing Property for Reaction-Diffusion Systems}}

The approach to proving the squeezing property for an evolutionary PDE is to study the difference of two solutions, $w(t) = g(t) - h(t)$, and conduct estimates to bound the time derivatives of the lower and higher modes, $d\|Pw\|/dt$ and $d\|Qw\|/dt$.

Consider a general system of reaction-diffusion equations in the form of an evolutionary equation on a real Hilbert space $H= L^2 (\gw, \mathbb{R}^d)$, where the higher dimensional $\gw \subset \mathbb{R}^d \,(d \geq 3)$, is a bounded Lipschitz domain,
\beq \label{ea00}
	\frac{dg}{dt} + \mathcal{A} g = f(g)
\eeq
where $f \in C^1 (\mathbb{R}^d, \mathbb{R}^d)$ is a nonlinear vector function and the differential operator $\mathcal{A}: \mathscr{D} (\mathcal{A}) \to H$ is a densely defined, nonnegative self-adjoint operator with compact resolvent so that its spectrum consists of a nonnegative sequence of the eigenvalues $\{\lambda_m\}$ with finite multiplicities and $\lambda_m \to \infty$ as $m \to \infty$.

Assume that the weak solution $g(t)$ of the evolutionary equation \eqref{ea00} exists on the time interval $[0, \infty)$ for any initial data $g_0 \in H$, such that
\beq \label{solg}
	g \in C([0, \infty), H) \cap L^2_{loc} ([0, \infty), E)
\eeq
where $E = H^1 (\gw, \mathbb{R}^d)$ whose norm is defined by $\|u\|^2_E = \|\nb u\|^2 + \| u \|^2$. Suppose that there exists a positively invariant, closed and bounded set $M \subset E$ for the solution semiflow such that
\beq \label{ea0}
	\|f(g) - f(\tg)\|_{H}  \leq C \|g - \tg\|_{E}, \quad \text{for any}\;\, g, \tg \in M,
\eeq
where the positive Lipschitz constant $C = C(M) > 0$, and 
\beq \bl{mp}
	\langle f(g) - f(\widetilde{g}), \, g - \widetilde{g} \rangle_H \leq C^* \|g - \widetilde{g}\|_H^2, \quad  \text{for any}\;\,g, \tg \in M,
\eeq
where $C^* > 0$ is a constant independent of $M$.

Let the complete set of the orthonormal eigenvectors of $\mathcal{A} : \mathscr{D}(\mathcal{A}) \rightarrow H$ associated with the eigenvalues $\{\lambda_i\}$ (each repeated to the respective multiplicity) be $\{e_i\}$,  $\mathcal{A} e_i = \lambda_i e_i$ and $ \lambda_i \leq \lambda_{i+1} \rightarrow \infty$. Let $P_m : H \rightarrow \text{Span} \; \{e_1,..., e_m\}$ and $Q_m = I - P_m$ be the orthogonal spectral projections. Then 
\begin{align*}
	\|p\|_E &= \left(\sum_{k = 1}^{m}\, |\inpt{p, e_k}|^2 \lambda_k \right)^{\frac{1}{2}} \leq \left(\lambda_m^{\frac{1}{2}} + 1\right) \|p\|_H, \quad \quad  p \in PH,  \\
	\|q\|_E &= \left(\sum_{k = m + 1}^{\infty} |\inpt{q, e_k}|^2 \lambda_k \right)^{\frac{1}{2}} \leq \left(\lambda_{m + 1}^{\frac{1}{2}} + 1\right) \|q\|_H, \quad  q \in QH, 
\end{align*}
where we briefly write $P = P_m$ and $Q = Q_m = I - P_m$.

We now prove a theorem on the squeezing property for the abstract reaction-diffusion equation \eqref{ea00} on a higher dimensional bounded domain. 

\begin{theorem} \label{SP}
	Under the assumptions \eqref{solg}, \eqref{ea0} and \eqref{mp}, there exists an integer $m \geq 1$ sufficiently large such that the squeezing property is satisfied on the compact, positively invariant and bounded set $M \subset H$  with respect to the projection mapping $P = P_m$ for the solution semiflow of the reaction-diffusion system \eqref{ea00}. 
	
\begin{proof}
		For two solutions $g(t)$ and $h(t)$ of \eqref{ea00} in the positively invariant set $M$, 
the difference $ \xi (t) = g(t)  - h(t)$ satisfies the equation
		\begin{equation}\label{ea1}
		\frac{d\xi}{dt} + \mathcal{A} \xi = f(g) - f(h), \quad t\geq 0.
		\end{equation}
Write $p(t) = P \xi (t)$ and $q(t) = Q \xi (t)$ so that $\xi (t) = p(t) + q(t)$ is an orthogonal decomposition of $\xi (t)$. Note that the closed and bounded set $M \subset E$ in the assumptions \eqref{ea0} and \eqref{mp} is a compact set in the space $H$. 
		
	Step 1. Take $L^2$ inner-product $\inpt{\eqref{ea1}, p(t)}$ and note that $\mathcal{A} P = P \mathcal{A}$ on $\mathscr{D}(\mathcal{A})$ and $P^2 = P$. We have
			\begin{equation*}
			\frac{1}{2} \frac{d}{dt} \|p(t)\|^2 + \|\nb p\|^2 = \inpt{f(g) - f(h), p} \geq - C \| \xi \|_E \|p\| \geq - C(\lambda_m^{\frac{1}{2}} + 1) \|\xi \| \|p\|
			\end{equation*}
			due to the Lipschitz condition \eqref{ea0}.
			Then
			\beq \label{eam2}
			\begin{split}
				\frac{1}{2} \frac{d}{dt} \|p(t)\|^2 &\geq - \lambda_m \|p\|^2 - C(\lambda_m^{\frac{1}{2}} + 1) (\|p\| + \|q\|) \|p\|\\
				& = -(\lambda_m + C (\lambda_m^{\frac{1}{2}} + 1)) \|p\|^2 - C(\lambda_m^{\frac{1}{2}} + 1) \|p\| \|q\|.
			\end{split}
			\eeq
			On the other side, we take the inner product $\inpt{\eqref{ea1}, q(t)}$ and obtain
			\beq \label{eam3}
			\begin{split}
				\frac{1}{2} \frac{d}{dt} \|q(t)\|^2 &\leq - \lambda_{m + 1} \|q\|^2 + C(\lambda_m^{\frac{1}{2}} + 1) (\|p\| + \|q\|) \|q\|\\
				&\leq -(\lambda_m - C(\lambda_m^{\frac{1}{2}} + 1)) \|q\|^2 + C(\lambda_m^{\frac{1}{2}} + 1) \|p\| \|q\|.
			\end{split}
			\eeq
			We choose $m$ sufficiently large such that 
			\beq \label{eam4}
			\lambda_m - C(\lambda_m^{\frac{1}{2}} + 1) > 2 C(\lambda_m^{\frac{1}{2}} + 1),
			\eeq
			
	Let $S = S(1)$ for $t^* = 1$. Then either 
	$$
		\|(I - P)(Sg - Sh)\| \leq \|P(Sg - Sh)\|, \quad \text{i.e.} \quad  \|q(1)\| \leq \|p(1)\|, 
	$$
	or otherwise
			\beq \label{eam5}
			\|(I - P) \xi (1) \| = \|Q\xi (1)\| > \|P\xi (1)\|, \quad \text{i.e.} \quad \|q(1)\| > \|p(1)\|.
			\eeq
	Below we consider the case that \eqref{eam5} occurs. By the choice \eqref{eam4}, we have
	$$
			(\lambda_m - C(\lambda_m^{\frac{1}{2}} + 1)) \|Q\xi (1)\| > 2 C(\lambda_m^{\frac{1}{2}} + 1) \|P\xi (1)\|.
	$$
	Namely,    
			\beq \label{eam6}
			(\lambda_m - C(\lambda_m^{\frac{1}{2}} + 1)) \|q (1)\| > 2 C(\lambda_m^{\frac{1}{2}} + 1) \|p (1)\|.
			\eeq
	The continuity of $\xi (t)$ in $H$ implies that the strict inequality as above holds for $t$ in a small neighborhood of $t^* = 1$. There are two possibilities to be considered.
			
	Step 2. The first possibility is that 
	\beq \bl{pqt}
		(\lambda_m - C(\lambda_m^{\frac{1}{2}} + 1)) \|q (t)\| > 2 C(\lambda_m^{\frac{1}{2}} + 1) \|p (t)\|
	\eeq
	holds for all $t \in \left[\frac{1}{2}, 1 \right]$. Then 
	\begin{equation} \bl{m6}
		\begin{split}
		&(\lambda_m - C(\lambda_m^{\frac{1}{2}} + 1)) \|q(t)\| - C(\lambda_m^{\frac{1}{2}} + 1) \|p(t)\|   \\
		> \frac{1}{2} (\lambda_m &- C(\lambda_m^{\frac{1}{2}} + 1)) \|q(t)\| > \frac{\lambda_m}{3} \|q(t)\|, \quad \text{for} \;\;  t \in [1/2, 1], 
		\end{split}
	\end{equation}
			where we used \eqref{eam6} in the first inequality and \eqref{eam4} in the second inequality of \eqref{m6}. Then \eqref{eam3} becomes
			$$
			\frac{d}{dt} \|q\|^2 \leq - \, \frac{2}{3}\, \lambda_m \|q\|^2, \quad t \in [1/2, 1].
			$$
			Integrating this inequality over the time interval $[\frac{1}{2}, 1]$, we obtain 
			$$
			 \|q(1)\|^2 \leq e^{-\lambda_m/3}\, \|q(1/2)\|^2.
			$$
			Since $\|\xi (1)\|^2 = \|p(1)\|^2 + \|q (1)|^2 \leq 2 \|q(1)\|^2$ due to \eqref{eam5}, it infers that 
			\beq \bl{ie12} 
			\|\xi (1)\| \leq \sqrt{2} \,\|q(1)\| \leq \sqrt{2}\, e^{- \lambda_m/6} \|q(1/2)\| \leq \sqrt{2} e^{- \lambda_m/3} \|\xi (1/2)\|.
			\eeq
			
	On the other hand, taking the inner product $\inpt{\eqref{ea1}, \xi (t)}$ and using the monotone property \eqref{mp}, we can get
			\begin{equation*}
			\frac{1}{2} \frac{d}{dt} \|\xi \|^2 \leq \frac{d}{dt} \|\xi \|^2 + \|\nb \xi \|^2 \leq \inpt{f(g) - f(h), g - h} \leq C^*\|g - h\|^2 = C^* \|\xi \|^2.
			\end{equation*}
	Integrate the above inequality over the time interval $[0, t]$, we get
			\beq \label{ea7}
			\|g(t) - h(t)\| \leq e^{C^* t} \|g_0 - h_0\|, \quad \text{for any} \;\;  t \geq 0.
			\eeq
			It yields, in particular, 
			\beq \bl{ie20} 
			\|\xi (1/2)\| \leq e^{C^* /2} \|\xi (0)\|,
			\eeq
	Then \eqref{ie12} and \eqref{ie20} give rise to the inequality
			\beq \label{ea8}
			\|S g_0 - S h_0\| = \|S(1)g_0 - S(1) h_0\| = \|\xi (1)\| \leq \delta \|\xi (0)\| = \delta \|g_0 - h_0\|
			\eeq
			with 
			\beq \bl{dlt}
			0 < \delta = \sqrt{2}\, e^{- \lambda_m/6} \, e^{C^* /2}  < 1
			\eeq
		        provided that $m$ is large enough so that $\lambda_m$ is large enough. Thus it is proved that for this first possibility the squeezing property is satisfied by the mapping $S$ and by the solution semiflow 				$\{S(t)\}_{t \geq 0}$ of the reaction-diffusion system \eqref{ea00}.
			
	Step 3. The second possibility is that \eqref{pqt} does not hold for all $t \in [1/2, 1]$. Then there is a time $\frac{1}{2} < t_0 < 1$ such that \eqref{pqt} is valid for $t \in (t_0, 1]$ and
			\beq \label{ea9}
			(\lambda_m - C(\lambda_m^{\frac{1}{2}} + 1)) \|q (t_0)\| = 2 C(\lambda_m^{\frac{1}{2}} + 1) \|p (t_0)\|.
			\eeq
			Define a function
			\beq \label{ea10}
			\Phi (t) = (\|p(t)\| + \|q(t)\|) \; \text{exp} \left(\frac{\lambda_m \|q(t)\|}{C_m (\|p(t)\| + \|q(t)\|)}\right)
			\eeq
			where $C_m = C (\lambda_m^{\frac{1}{2}} + 1)$. From \eqref{eam2} and \eqref{eam3}, since $\frac{1}{2} \frac{d}{dt} \|p(t)\|^2 = \|p(t)\| \frac{d}{dt} \|p(t)\|$ and similarly for $\|q(t)\|$, we have
			\begin{gather*}
			\frac{d}{dt} \|p\| \geq -(\lambda_m + C(\lambda_m^{\frac{1}{2}} + 1))\|p\| - C (\lambda_m^{\frac{1}{2}} + 1)\|q\|,  \\
			\frac{d}{dt} \|q\| \leq -(\lambda_m - C(\lambda_m^{\frac{1}{2}} + 1)) \|q\| + C (\lambda_m^{\frac{1}{2}} + 1)\|p\|.
			\end{gather*} 
			Then
			\beq \label{ea10a}
			\frac{d}{dt} \Phi(t) = \exp \left[\frac{\lambda_m \|q\|}{C_m (\|p\| + \|q\|)} \right] \left[\frac{d}{dt} (\|p\| + \|q\|) + (\|p\| + \|q\|) \frac{d}{dt}\left(\frac{\lambda_m \|q\|}{C_m (\|p\| + \|q\|)} \right)\right].
			\eeq
			Since the exponential factor is positive, in order to know the sign of the derivative $\frac{d}{dt} \Phi (t)$, we only need to estimate the second factor on the right side of \eqref{ea10a}:
			\begin{equation} \bl{dph}
			\begin{split}
			& \frac{d}{dt} (\|p\| + \|q\|) + (\|p\| + \|q\|)\, \frac{d}{dt} \left(\frac{\lambda_m \|q\|}{C_m (\|p\| + \|q\|)}\right) \\
			=  &\,\frac{d}{dt} \|p\| + \frac{d}{dt} \|q\| + \frac{\lambda_m}{C_m} \frac{d}{dt} \|q\| - \frac{\lambda_m \|q\|}{C_m (\|p\| + \|q\|)} \left(\frac{d}{dt} \|p\| + \frac{d}{dt}\|q\|\right)\\
			=  &\, \frac{d}{dt}\|q\|\left[1 + \frac{\lambda_m}{C_m} - \frac{\lambda_m \|q\|}{C_m (\|p\| + \|q\|)}\right] - \frac{d}{dt} \|p\| \left[\frac{\lambda_m \|q\|}{C_m (\|p\| + \|q\|)} -1 \right]\\
			=  &\, \frac{d}{dt}\|q\| \left[1 + \frac{\lambda_m \|p\|}{C_m (\|p\| + \|q\|)}\right] - \frac{d}{dt} \|p\| \left[\frac{\lambda_m \|q\|}{C_m (\|p\| + \|q\|)} - 1 \right]\\
			\leq &\, ( -(\lambda_m - C_m) \|q\| + C_m \|p\|) \left[1 + \frac{\lambda_m \|p\|}{C_m (\|p\| + \|q\|)}\right]\\
			 +  &\, ((\lambda_m + C_m) \|p\| + C_m \|q\|) \left[\frac{\lambda_m \|q\|}{C_m (\|p\| + \|q\|)} -1 \right]\\
			=  &\, -(\lambda_m - C_m) \|q\| - \frac{\lambda_m (\lambda_m - C_m) \|p| \|q\|}{C_m (\|p\| + \|q\|)} + C_m \|p\| + \frac{\lambda_m \|p\|^2}{\|p\| + \|q\|} \\
			+ &\, \frac{\lambda_m (\lambda_m + C_m) \|p \| \|q\|}{C_m (\|p\| + \|q\|)} - (\lambda_m + C_m) \|p\| + \frac{\lambda_m \|q\|^2}{\|p\| + \|q\|} - C_m \|q\|\\
			=  &\, -\lambda_m \|q\| - \lambda_m \|p\| + \frac{2 \lambda_m \|p\| \|q\|}{\|p\| + \|q\|} + \frac{\lambda_m \|p\|^2}{\|p\| + \|q\|} + \frac{\lambda_m \|q\|^2}{\|p\| + \|q\|}\\
			=  &\, - \lambda_m \|q\| - \lambda_m \|p\| + \frac{\lambda_m (\|p\|^2 + \|q\|^2 + 2\|p\| \|q\|)}{\|p\| + \|q\|} = 0.
			\end{split}
			\end{equation}	
			Hence we obtain
			$$ \frac{d}{dt} \, \Phi(t) \leq 0, \quad \text{for} \; t \in [t_0,1].$$
			It follows that 
			\beq \label{ea11}
			\Phi(1) \leq \Phi(t_0).
			\eeq
			
	At $t = 1, \|q(1)\| = \|Q \xi (1)\| > \|P \xi (1)\| = \|p(1) \|$ by \eqref{eam5}. Then from \eqref{ea10} we see that  
			\beq \label{ea12}
			\Phi(1) \geq \|q(1)\| \; \text{exp} \left(\frac{\lambda_m \|q(1)\|}{2 C_m \|q(1)\|} \right) = \|q(1)\| e^{\lambda_m/(2 C_m)}
			\eeq
			At $t = t_0$, \eqref{ea9} indicates that
			$$ (\lambda_m - C_m) \|q(t_0)\| = 2 C_m \|p(t_0)\|$$
			and then
			$$ 2 C_m (\|p(t_0)\| + \|q(t_0)\|) = (\lambda_m + C_m) \|q(t_0)\|.$$
			Thus,
			\beq \label{ea13}
			\Phi(t_0) = \frac{\lambda_m + C_m}{2 C_m}\, \|q(t_0)\| \, \exp \left(\frac{2 \lambda_m}{\lambda_m + C_m} \right).
			\eeq
	Note that $t_0 \in (1/2, 1]$. Put together \eqref{ea11}, \eqref{ea12} and \eqref{ea13}. We use the Lipschitz continuous dependence on initial data to obtain
			\begin{equation} \bl{qxi}
			\begin{split}
			\|q(1)\| &\leq \exp \left( - \frac{\lambda_m}{2 C_m}\right) \Phi(1) \leq \exp \left( - \frac{\lambda_m}{2 C_m}\right) \Phi(t_0)  \\
			&\leq \exp \left( - \frac{\lambda_m}{2 C_m}\right) \frac{\lambda_m + C_m}{2 C_m} \exp \left(\frac{2 \lambda_m}{\lambda_m + C_m} \right) \|q (t_0)\|.  \\
			&\leq \exp \left( - \frac{\lambda_m}{2 C_m}\right) \frac{\lambda_m + C_m}{2 C_m} \, e^2 \, \|q (t_0)\|   \\
			&\leq \exp \left( - \frac{\lambda_m}{2 C_m}\right) \frac{\lambda_m + C_m}{2 C_m} \; e^2 \, \|\xi (t_0)\|.
			\end{split}
			\end{equation}
	According to the solution expression of the evolutionary equation \eqref{ea1},
			$$
			\xi (t) = e^{- \mathcal{A} t} \xi (0)  +  \int_0^t e^{- \mathcal{A} (t-s)} (f(g(s)) - f(h(s)))\, ds \quad t \geq 0,
			$$
			By using the Lipschitz condition \eqref{ea0} and the fact that $e^{-\mathcal{A} t}$ is a contraction semigroup, we can deduce that 
			\beq \bl{xit}
				\begin{split}
				&\| \xi (t) \| \leq \|e^{- \mathcal{A} t}\|_{\mathcal{L}(H)} \|\xi (0)\| + \int_0^t \|e^{- \mathcal{A}(t-s)}\|_{\mathcal{L}(H)} \|f(g(s)) - f(h(s))\|\, ds \\
				\leq &\, \|\xi (0)\| + \int_0^t C \|g(s) - h(s)\|_E\, ds \leq \|\xi (0)\| + \int_0^t C \| \xi (s)\|_E\, ds, \quad t \geq 0.
				\end{split}
			\eeq
			Then the Gronwall inequality applied to \eqref{xit} shows that
			
			$$
				\|\xi (t)\| \leq \|\xi (0) \|\, e^{Ct},  \quad t \geq0.
			$$
			
			\vspace{5pt}
			Substitute this inequality at $t_0$ into \eqref{qxi} to obtain
			\begin{equation*} 
			\begin{split}
			\|q(1) \| &\leq \exp \left( - \, \frac{\lambda_m}{2 C_m}\right) \frac{\lambda_m + C_m}{2 C_m} \; e^2 \, \|\xi (t_0)\|  \\[5pt]
			&\leq \exp \left( - \, \frac{\lambda_m}{2 C_m}\right) \frac{\lambda_m + C_m}{2 C_m} \; e^{2 + C}\, \|\xi (0)\|.
			\end{split}
			\end{equation*}
			Since $\|p(1)\| < \|q(1)\|$, we end up with
			$$ 
			\|\xi (1)\| \leq \sqrt{2} \, \exp \left( - \, \frac{\lambda_m}{2 C_m}\right) \frac{\lambda_m + C_m}{2 C_m} \; e^{2 + C} \|\xi (0)\|. 
			$$
			For $m$ sufficiently large, we can assert 
			\begin{equation*}
			\begin{split} 
			0 < \delta &= \sqrt{2} \, \exp \left(- \, \frac{\lambda_m}{2 C_m}\right) \frac{\lambda_m + C_m}{2 C_m} \; e^{2 + 2C}   \\
		& = \sqrt{2} \, \left(\frac{\lambda_m}{2 C(\lambda_m^{\frac{1}{2}} + 1)} + \frac{1}{2}\right) \exp \left(- \,\frac{\lambda_m}{2 C(\lambda_m^{\frac{1}{2}} + 1)}\right) e^{2 + 2C} < 1.
			\end{split}
			\end{equation*}
			We have proved that 
			\beq \label{ea14}
			\|Sg_0 - Sh_0\| = \|\xi (1)\| \leq \delta \|\xi (0)\| = \delta \|g_0 - h_0\|, \quad \text{for any} \;\, g_0, h_0 \in M.
			\eeq
			
Finally \eqref{ea8} and \eqref{ea14} show that, in any case as we have treated in Step 2 and Step 3, if the spectral number $m$ of the finite-rank orthogonal projection $P_m$ on the space $H$ is chosen to be large enough, then the squeezing property holds for the Hindmarsh-Rose semiflow $\{S(t)\}_{t \geq 0}$ generated by the equation \eqref{ea00} on this compact, positively invariant and bounded set $M \subset H$. The proof is completed. 				
	\end{proof}		
\end{theorem}

\section{\textbf{The Existence of Exponential Attractor}}

In this section, we shall prove the main result on the existence of an exponential attractor for the solution semiflow of the diffusive Hindmarsh-Rose equations. 

We start with the squeezing property stated Theorem \ref{SP} and check its two conditions \eqref{ea0} and \eqref{mp} are satisfied by the Hindmarsh-Rose semiflow.

\begin{lemma} \bl{LpHc}
	Under the same assumptions as in Theorem \ref{MTh}, the Nemytskii operator $f$ defined by \eqref{opf} satisfies the $E$ to $H$ Lipschitz condition
\beq \bl{fHE}
	\|f(g) - f(\tg)\|_H \leq C_E (M) \|g - \tg \|_E, \quad \text{for any} \;\; g, \tg \in M,
\eeq
on any given positively invariant and bounded set $M \subset E$, where $C_E(M) > 0$ is a constant only depending on $M$. Moreover, $f$ satisfies the monotone property that there exists a constant $C^* > 0$ independent of $M$ such that 
\beq \bl{fmn}
	\langle f(g) - f(\tg), g - \tg \rangle \leq C^* \|g - \tg \|^2, \quad   \text{for any} \;\;  g, \, \tg \in M.
\eeq
\end{lemma}
\begin{proof}
	First we prove the claim \eqref{fmn}. For any $g = (u, v, w)$ and $\Tilde{g} = (\tu, \tv, \tw)$ in the set $M$ and denote the three components of $f$ by $f_1, f_2, f_3$. For the first component $f_1$, we have
\beq \bl{f1}
	\begin{split}
	&\langle f_1(g) - f_1 (\tg), u - \tu \rangle =  \langle f_1(u, v, w) - f_1 (\tu, \tv, \tw), u - \tu \rangle   \\[7pt]
	\leq &\, \langle \vp(u) - \vp(\tu), u - \tu \rangle + \langle v - \tv, u - \tu \rangle + \langle w - \tw, u - \tu \rangle \\[7pt]
	\leq &\, a \langle u^2 - \tu^2, u - \tu \rangle  - b \langle u^3 - \tu^3, u - \tu \rangle + \|v - \tv \| \|u - \tu \| + \| w - \tw \| \|u - \tu \| \\
	\leq &\, a \int_\Omega |u(x) - \tu (x)|^2 (u(x) + \tu (x))\, dx \\
	- &\, b \int_\Omega |u(x) - \tu (x)|^2 (u^2 (x) + u(x) \tu(x) + \tu^2 (x)) dx  \\[4pt]
	+ &\, \|u - \tu \|^2 + \|v - \tv \|^2 + \|w -\tw \|^2 \\[4pt]
	\leq &\ \frac{2a^2}{b} \int_\Omega |u(x) - \tu (x)|^2 \, dx + \|u - \tu \|^2 + \|v - \tv \|^2 + \|w -\tw \|^2 \\
	- &\, \frac{b}{4} \int_\Omega |u(x) - \tu (x)|^2 (u^2 (x) + \tu^2 (x))\, dx, 
	\end{split}
\eeq
where we used 
$$
	a (u(x) + v(x)) \leq \frac{b}{4} (u^2 (x) + \tu^2 (x)) + \frac{2a^2}{b}.
$$
For the second component $f_2$, we do the estimate
\beq \bl{f2}
	\begin{split}
	&\langle f_2(g) - f_2 (\tg), v - \tv \rangle =  \langle f_2(u, v, w) - f_2 (\tu, \tv, \tw), v - \tv \rangle   \\[7pt]
	\leq &\, \langle \psi (u) - \psi (\tu), v - \tv \rangle + \|v - \tv \|^2 \\
	= &\, \beta \int_\Omega (u(x) - \tu (x)) (u(x) + \tu (x)) (v(x) - \tv (x))\, dx + \|v - \tv \|^2 \\
	\leq &\, \frac{b}{8} \int_\Omega |u(x) - \tu (x)|^2 |u(x) + \tu (x)|^2 \, dx + \frac{2\beta^2}{b} \|v - \tv \|^2 + \|v - \tv \|^2 \\
	\leq &\, \frac{b}{4} \int_\Omega |u(x) - \tu (x)|^2 (u^2(x) + \tu^2 (x)) \, dx + \left(1 + \frac{2\beta^2}{b}\right) \|v - \tv \|^2.
	\end{split}
\eeq
For the third component $f_3$, we have
\beq \bl{f3}
	\begin{split}
	&\langle f_3(g) - f_3 (\tg), w - \tw \rangle =  \langle f_3(u, v, w) - f_3 (\tu, \tv, \tw), w - \tw \rangle \\[3pt]
	\leq &\, q \|u - \tu \| \|w - \tw \| + r \|w - \tw \|^2 \leq q \|u - \tu \|^2 + (q + r) \|w - \tw \|^2.
	\end{split}
\eeq
Summing up \eqref{f1}, \eqref{f2} and \eqref{f3} with two integral terms on the right-hand sides being cancelled out, we obtain
\beq \bl{fub}
	\begin{split}
	&\langle f(g) - f(\tg), g - \tg \rangle = \langle f_1(g) - f_1 (\tg), u - \tu \rangle \\[7pt]
	+ &\, \langle f_2(g) - f_2 (\tg), v - \tv \rangle + \langle f_3(g) - f_3 (\tg), w - \tw \rangle \\
	\leq &\, \left( 1 + \frac{2a^2}{b} \right) \|u - \tu\|^2 + \|v - \tv\|^2 + \|w - \tw \|^2 \\
	+ &\, \left(1 + \frac{2\beta^2}{b}\right) \|v - \tv \|^2 + q \|u - \tu \|^2 + (q + r) \|w - \tw \|^2 \\
	= &\,  \left( 1 + q + \frac{2a^2}{b} \right) \|u - \tu\|^2 +  \left(2 + \frac{2\beta^2}{b}\right) \|v - \tv\|^2 + (1 + q + r) \|w - \tw \|^2 \\
	\leq &\, C^* \, (\|u - \tu\|^2 +  \|v - \tv\|^2 +  \|w - \tw \|^2) = C^* \, \|g - \tg\|^2.
	\end{split}
\eeq
Thus the inequality \eqref{fmn} is satisfied by $f$ on the set $M$ with a uniform constant
\beq \bl{Cp}
	C^* = \max \left\{1 + q + \frac{2a^2}{b}, \, 2 + \frac{2\beta^2}{b}, \, 1 + q + r \right\}.
\eeq

Next we prove the $E$ to $H$ Lipschitz condition \eqref{fHE} of the Nemytskii operator $f$. Due to the Sobolev embedding $E = H^1 (\gw, \mathbb{R}^3) \hookrightarrow L^6 (\gw, \mathbb{R}^3) \hookrightarrow L^4 (\gw, \mathbb{R}^3)$, there are positive constants $\ggd_1$ and $\ggd_2$ such that 
$$
	\|\cdot \|_{L^4 (\Omega)} \leq \ggd_1 \| \cdot \|_{H^1 (\Omega)} \quad \text{and} \quad \|\cdot \|_{L^6 (\Omega)} \leq \ggd_2 \| \cdot \|_{H^1 (\Omega)}. 
$$
Since $M$ is an invariant and bounded set in $E$, we define
$$
	N_1 = \max_{g \in M} \|u\|_{L^4}, \quad  N_2 = \max_{g \in M} \|u\|_{L^6}.
$$ 
Then
\beq \bl{fg1}
	\begin{split}
	\| f(g) - &\, f(\tg) \|^2_H = \| f_1(g) - f_1(\tg) \|^2 + \| f_2(g) - f_2(\tg) \|^2 + \| f_3(g) - f_3(\tg) \|^2 \\[3pt]
	\leq &\, (a\|u^2 - \tu^2\| + b \|u^3 - \tu^3\| + \|v - \tv \| + \|w - \tw \|)^2 \\[3pt]
	+ &\, (\beta \| u^2 -\tu^2\| + \|v - \tv\|)^2 + (q \|u - \tu\| + r \|w - \tw\|)^2 \\[2pt]
	\leq &\, 4(a^2\|u^2 - \tu^2\|^2 + b^2 \|u^3 - \tu^3\|^2 + \|v - \tv \|^2 + \|w - \tw \|^2) \\[3pt]
	+ &\, 2(\beta^2 \| u^2 -\tu^2\|^2 + \|v - \tv\|^2) + 2(q^2 \|u - \tu\|^2 + r^2 \|w - \tw\|^2) \\[3pt]
	= &\, (4a^2 + 2\beta^2) \|u^2 - \tu^2\|^2 + 4b^2 \|u^3 - \tu^3\|^2 + 2q^2 \|u - \tu\|^2 \\[3pt]
	+ &\, 6 \|v - \tv\|^2 + (4 + 2r^2)\|w - \tw \|^2.
	\end{split}
\eeq
Note that H\"{o}lder inequality implies that
\begin{align*}
	\|u^2 - \tu^2\|^2 = &\, \|(u - \tu)(u + \tu)\|^2 = \int_\Omega |u(x) - \tu (x)|^2 |u(x) + \tu (x)|^2 \, dx \\
	\leq &\, \|u - \tu\|_{L^4}^2 \|u + \tu\|_{L^4}^2 \leq 4\,\ggd_1^2 N_1^2 \|u - \tu\|_{H^1 (\Omega)}^2
\end{align*}
and
\begin{align*}
	&\|u^3 - \tu^3\|^2 = \|(u - \tu)(u^2 + u \tu + \tu)^2 \|^2 \\[3pt]
	= &\, \int_\Omega |u(x) - \tu (x)|^2 |u^2(x) + u(x)\tu(x) + \tu^2(x)|^2 \, dx \\
	\leq &\, \left(\int_\Omega |u(x) - \tu (x)|^6\, dx\right)^{1/3} \left(\int_\Omega |u^2 (x) + u(x)\tu (x) + \tu^2 (x)|^3\, dx\right)^{2/3} \\[2pt] 
	= &\, \|u - \tu \|_{L^6}^2 \| u^2 + u \tu + \tu^2 \|_{L^6}^4 \leq 4 \ggd_2^2 \|u - \tu \|_{H^1}^2 \| 2u^2 + 2\tu^2 \|_{L^3}^2 \\[6pt]
	\leq &\, 4 \ggd_2^2 \|u - \tu \|_{H^1}^2 \cdot (4 \|u\|_{L^6}^4 + 4 \|\tu\|_{L^6}^4) \leq 32\, \ggd_2^2 \,N_2^4 \|u - \tu \|_{H^1 (\Omega)}^2.
\end{align*}
Substitute the above two inequalities into \eqref{fg1}. We obtain
\beq \bl{fg2}
	\begin{split}
	\| f(g) - f(\tg) \|_H^2 \leq &\, \left(4\,\ggd_1^2 N_1^2 (4a^2 + 2\beta^2 ) + 128\, b^2 \ggd_2^2 \,N_2^4  + 2q^2 \right)  \|u - \tu \|_{H^1 (\Omega)}^2 \\[3pt]
	+&\, 6 \|v - \tv\|^2 + (4 + 2r^2)\|w - \tw \|^2 
	\end{split}
\eeq
which shows that \eqref{fHE} is valid with the constant $C_E (M) > 0$ given by
$$
	C_E (M) = \sqrt{\max \left\{4\,\ggd_1^2 N_1^2 (4a^2 + 2\beta^2 ) + 128\, b^2 \ggd_2^2 \,N_2^4  + 2q^2, \; 6, \, 4 + 2r^2 \right\}}.
$$
The proof is completed.
\end{proof}

	Now we prove the existence of an exponential attractor for the Hindmarsh-Rose semiflow $\{S(t)\}_{t \geq 0}$ generated by the Hindmarsh-Rose evolutionary equation \eqref{pb}

\begin{theorem} \label{EA}
	Under the same assumptions as in Theorem \ref{MTh}, there exists an exponential attractor $\mathscr{E}$ in the space $H = L^2 (\gw, \mathbb{R}^3)$ for the Hindmarsh-Rose semiflow $\{S(t)\}_{t \geq 0}$ generated by the weak solutions of the diffusive Hindmarsh-Rose equations \eqref{pb}. 
\end{theorem}

\begin{proof}
We shall go through the following steps to check all the three conditions stated in Theorem \ref{ExpAr}. 

Step 1. First we show that there exists a compact, positively invariant and absorbing set $M \subset H$ for the Hindmarsh-Rose semiflow $\{S(t)\}_{t \geq 0}$ such that \eqref{ea0} and \eqref{mp} are satisfied. Then according to Theorem \ref{SP} the squeezing property is satisfied for the mapping $S(t^*)$ at $t^* = 1$ on this set $M$.  

Theorem \ref{AbE} has shown that the closed and bounded ball $B_E (Q)$ centered at the origin with radius $Q > 0 $ in the space $E = H^1 (\gw, \mathbb{R}^3)$ is an absorbing set for this semiflow. We can easily verify that the set
\beq \label{SetM}
	M = \overline{\bigcup_{0 \leq t \leq T^*}  S(t) B_E (Q)}
\eeq
is a compact, positively invariant and absorbing set in the space $H$ for this semiflow, where $T^* = T^* (B_E (Q))$ is the permanently entering time for the solution trajectories starting from the ball $B_E (Q)$ into itself, as indicated in \eqref{ac}. The compactness of $M$ in $H$ is inferred by the boundedness of $M$ in the space $E$ and the compact embedding $E \hookrightarrow H$ so that the cylinder $[0, T^*] \times B_E (Q)$ is a compact set in $\mathbb{R} \times H$ and by the fact that the function 
\beq \bl{Stg}
	\gamma (t, g) = S(t) g \,\; \text{is continuous on} \,\; [0, T^*] \times B_E (Q).
\eeq

In Lemma \ref{LpHc} it has been shown that the nonlinear mapping $f(g)$ given in \eqref{opf} satisfies the Lipschitz continuous condition \eqref{ea0} and the monotone condition \eqref{mp} on this  set $M$ given in \eqref{SetM}. Moreover by the continuity of the functions $\ga (t, g)$ in \eqref{Stg}, we see that 
\beq \bl{gabd}
	G = \max \{\|\ga (t, g)\|_E: (t, g) \in [0, T^*] \times B_E (Q)\} < \infty.
\eeq

Thus we can apply Theorem \ref{SP} with its proof to confirm that the squeezing property is satisfied by the mapping $S(t^*)$ at $t^* = 1$ so that the squeezing property is satisfied by the Hindmarsh-Rose semiflow $\{S(t)\}_{t \geq 0}$ on this set $M$ in $H$. Therefore, the first condition in Theorem \ref{ExpAr} is satisfied by the Hindmarsh-Rose semiflow.

        Step 2. Next we show that, for the Hindmarsh-Rose semiflow and for any $t \in [0, t^*] = [0, 1]$, the mapping $S(t): M \to M$ is Lipschitz continuous in $H$ and the associated Lipschitz constant $K(t): [0, 1] \to (0, \infty)$ is a bounded function. 

For this purpose, consider any two $g_0 = (u_0, v_0, w_0), \tg_0 = (\tu_0, \tv_0, \tw_0) \in M$ and the solutions $g(t) = S(t) g_0$ and $\tg (t) = S(t) \tg_0$ for $t \in [0, 1]$. Then $h(t) = g(t) - \tg (t)$ satisfies the equation
\beq \label{gg}
	\begin{split}
	\frac{dh}{dt} &\, = Ah + f(g) - f(\tg), \quad t > 0, \\
	& h(0) = h_0 = g_0 - \tg_0.
	\end{split}
\eeq
The three component functions of $h(t) = (U(t), V(t), W(t))$ can be estimated as follows. First,

\beq \label{Uq}
	\begin{split}
	&\frac{1}{2} \frac{d}{dt} \|U\|^2 + d_1\|\nb U\|^2 =  \langle f_1 (g) - f_1 (\tg), u - \tu \rangle \\[6pt]
	= &\,  \langle (\vp (u) - \vp (\tu)) + (v  - \tv) - (w - \tw), u - \tu \rangle \\[5pt]
	= &\, \int_\gw \left(a(u^2 -  \tu^2) - b(u^3 - \tu^3) + (v  - \tv) - (w - \tw)\right) (u - \tu)\, dx \\
	= &\,  \int_\gw \left(a(u -  \tu)^2 (u + \tu) - b(u - \tu)^2 (u^2 \tu + u \tu + \tu^2)\right) dx \\
	&\, + \int_\gw ((v  - \tv)(u - \tu) - (w - \tw)(u - \tu)) \, dx \\
	\leq &\, \int_\gw (u -  \tu)^2 \left[ a(u  + \tu) - b (u^2 + u \tu + \tu^2)\right] dx \\[2pt]
	&\, + \| u - \tu \| (\|v - \tv \| + \|w - \tw \|) \\
	\leq &\, \int_\gw (u -  \tu)^2 \left[ a(u  + \tu) - b (u^2 + u \tu + \tu^2)\right] dx + 2 \|g - \tg \|^2
	\end{split}
\eeq
and by Young's inequality we have 
\begin{gather*}
	a (u + \tu) -  - b (u^2 + u \tu + \tu^2) = [a (u + \tu) - b u\tu] - b(u^2 + \tu^2) \\[5pt]
	\leq  \left(\frac{b}{4} u^2 + \frac{a^2}{b}\right) + \left(\frac{b}{4} \tu^2 + \frac{a^2}{b}\right) + \frac{b}{2} (u^2 + \tu^2) - b(u^2 + \tu^2) \leq - \frac{b}{4} (u^2 + \tu^2) + \frac{2a^2}{b}.
\end{gather*}
It follows that
\beq \label{Ue}
	\begin{split}
	& \frac{d}{dt} \|U\|^2 \leq \frac{d}{dt} \|U\|^2 + 2d_1\|\nb U\|^2 \\[5pt]
	\leq &\, 2\int_\gw (u - \tu)^2 \left( - \frac{b}{4} (u^2 + \tu^2) + \frac{2a^2}{b}\right) dx + 4 \|g - \tg\|^2 \\
	\leq &\, \int_\gw (u - \tu)^2 \left( - \frac{b}{2} (u^2 + \tu^2)\right) dx + \frac{4a^2}{b} \|u - \tu \|^2 + 4 \|g - \tg\|^2 \\
	\leq &\, - \frac{b}{2} \int_\gw (u - \tu)^2 (u^2 + \tu^2)\, dx + 4\left( 1+ \frac{a^2}{b}\right) \|h \|^2. 	
	\end{split}
\eeq
Similarly, for the second and third components of $h(t) = g(t) - \tg (t) = (U(t), V(t), W(t))$, we get

\beq \label{Ve}
	\begin{split}
	\frac{d}{dt} &\, \|V\|^2 \leq \frac{d}{dt} \|V\|^2 + 2d_2\|\nb V\|^2 \leq 2\langle \psi (u) - \psi (\tu) - (v - \tv), v - \tv \rangle \\[3pt]
	= &\, 2\int_\gw \left( - \beta (u^2- \tu^2) - (v  - \tv) \right) (v - \tv)\, dx \\
	\leq &\, 2\int_\gw \left(- \beta (u - \tu) u (v - \tv) - \beta (u - \tu)\tu (v - \tv)\right) dx \\
	\leq &\, \int_\gw \left(\frac{bu^2}{2} (u - \tu)^2 + \frac{b\tu^2}{2} (u - \tu)^2\right) dx + \frac{4\beta}{b} \|v - \tv \|^2  \\
	\leq &\; \frac{b}{2} \int_\gw \, (u^2 + \tu^2) (u - \tu)^2 \, dx + \frac{4\beta}{b} \|h \|^2
	\end{split}
\eeq
and
\beq \label{We}
	\begin{split}
	\frac{d}{dt} &\, \|W\|^2 \leq \frac{d}{dt} \|W\|^2 + 2d_3\|\nb W\|^2 \leq 2\langle q(u - \tu) - r(w - \tw), w - \tw \rangle \\[3pt]
	= &\, 2\int_\gw \left(q (u- \tu) - r(w  - \tw) \right) (w - \tw)\, dx \\[3pt]
	\leq &\, q \|u - \tu\|^2 + (q + 2r) \|w - \tw \|^2 \leq 2 (q + r) \|h \|^2.
	\end{split}
\eeq
Add up the inequalities \eqref{Ue}, \eqref{Ve} and \eqref{We} with a cancellation of the first terms on the rightmost side of \eqref{Ue} and \eqref{Ve}. Then we obtain
\beq \label{He}
	\frac{d}{dt} \|h \|^2 = \frac{d}{dt} \left(\|U\|^2 + \|V\|^2 + \|W\|^2 \right) \leq C_* \| h \|^2, \quad t > 0,
\eeq
where $C_*$ is a positive constant given by
$$
	C_* = 4\left( 1+ \frac{\beta}{b} + \frac{a^2}{b}\right) + 2 (q + r).
$$
Solve the differential inequality \eqref{He} to get
\beq \label{Kt}
	 \|g(t) -\tg (t)\| = \|h(t) \| \leq e^{C_* t/2} \|h(0)\| = K(t) \|g_0 - \tg_0 \|,  \quad t \geq 0,
\eeq
where $K(t) = e^{C_* t/2} \in [1, e^{C_*/2}]$ is a bounded function on the time interval $t \in [0, t^*],\, t^* = 1$. The claim at the beginning of this step is proved.

Step 3. Finally we show that for any given $g \in M$ the mapping $S(\cdot)g_0: [0, t^*] = [0, 1] \to M$ is Lipschitz continuous and the associated Lipschitz constant $L(g_0): M \to (0, \infty)$ is a bounded function.

For any given $g_0 \in M$, since the weak solution $S(t)g_0, t \geq 0$, is a mild solution for the evolutionary equation \eqref{pb}, we have
\beq \label{mds}
	S(t) g_0 = e^{At} g_0 + \int_0^t e^{A(t-s)} f(g(s, g_0))\, dt, \quad t \geq 0,
\eeq
where the operator $A$ and the nonlinear mapping $f$ are defined in \eqref{opA} and \eqref{opf}, respectively. Note that the parabolic semigroup $\{e^{At}\}_{t \geq 0}$ is a self-adjoint contraction semigroup so that  $\max_{t \geq 0} \|e^{At}\|_{\mathcal{L} (H)} = 1$. A fundamental theorem on sectorial operators \cite[Theorem 37.5]{SY} shows that the operator function $e^{At}: [0, \infty) \to \mathcal{L}(H)$ is uniformly Lipschitz continuous. Actually, the spectral expansion of $e^{At}$ shows
$$
	(e^{At} g_0)(x) = \sum_{k=1}^\infty \, e^{- \lambda_k t} \langle g_0, e_k \rangle e_k (x), \quad g \in H, \; t \geq 0, 
$$
where $\{- \lambda_k\}_{k=1}^\infty$, with $0 \leq \lambda_k \to \infty$ as $k \to \infty$, is the set of all the eigenvalues (repeated to the respective multiplicities) of $A: D(A) \to H$,  and $\{e_k\}_{k=1}^\infty$ with $Ae_k = - \lambda_k e_k$ is the complete set of the orthonormal eigenvectors of $A$. Then we can derive the Lipschitz continuity of $e^{At}$: For any $g_0 \in M$ and any $0 \leq \tau < t$,
\begin{equation} \bl{LpG}
	\begin{split}
	&\|e^{At} g_0 - e^{A\tau} g_0\|^2 =  \sum_{k=1}^\infty \,  |e^{- \lambda_k(t - \tau)}|^2 |\langle g_0, e_k \rangle |^2  \\
	= &\, \sum_{k=1}^\infty \,  |e^{- \zeta_k}|^2 \lambda_k |t - \tau| |\langle g_0, e_k \rangle |^2  \; \; (\text{where} \; 0 \leq \lambda_k \tau \leq \zeta_k \leq \lambda_k t) \\[5pt]
	\leq &\, |t - \tau| \|\nb g_0 \|^2 \leq  |t - \tau| \|g_0 \|_E^2 \leq G^2 |t - \tau |.
	\end{split}
\end{equation}

Therefore, we can deduce that, for any  $0 \leq \tau < t$,
\begin{equation} \label{Me}
	\begin{split}
	&\|S(t) g_0 - S(\tau) g_0\|_H \leq \|e^{At} g_0 - e^{A\tau} g_0\| + \int_\tau^t \|e^{A(t-s)} f(g(s, g_0))\| \, dt  \\
	\leq &\, G^2 |t - \tau|  + \int_\tau^t \|e^{A(t-s)}\|_{\mathcal{L} (H)} \|f(g(s, g_0))\|_H\, dt \\
	\leq &\, G^2 |t - \tau|  + \int_\tau^t  \|f(g(s, g_0)) - f(0)\|_H\, dt + \int_\tau^t \|f(0)\|_H\, dt  \\
	\leq &\, G^2 |t - \tau|  + \int_\tau^t  \, C_E (M)  \|g(s, g_0)\|_E\, dt + (J + \alpha + q|c|) |t - \tau| \\[8pt]
	\leq &\, G^2 |t - \tau| \| + C_E (M)G^2 |t - \tau| + (J + \alpha + q|c|) |t - \tau| \\[10pt]
	\leq &\, L(M) |t - \tau|,
	\end{split}
\end{equation}
where the Lipschitz constant $C_E (M)$ is given in \eqref{fHE} and 
$$
	L(M) = (1 + C_E (M)) G^2 + (J + \alpha + q|c|).
$$
Then clearly the claim in Step 3 is proved.

Since we have proved that all the three conditions in Theorem \ref{ExpAr} are satisfied by the Hindmarsh-Rose semiflow, there exists an exponential attractor $\mathscr{E}$ in the space $H$ for this Hindmarsh-Rose semiflow. The proof is completed.
\end{proof}

The existence of an exponential attractor as well as the squeezing property have the following meaningful corollaries on the finite fractal dimensionality of the global attractor shown in \cite{PYS} and on the determining modes. 

\begin{corollary} \label{Col1}
	The global attractor $\mathscr{A}$ of the Hindmarsh-Rose semiflow has a finite fractal dimension
\beq \label{dimA}
	\dim_F (\mathscr{A}) \leq N \max \left\{1, \frac{\log (2\sqrt{2} K/\theta + 1)}{- \log \theta} \right\}, \quad \theta \in (0, 1),
\eeq
where $N$ is the rank of the spectral projection associated with the squeezing property of the mapping $S(1)$ and $K$ is the Lipschitz constant of the mapping $S(1)$ on the compact, positively invariant, absorbing set $M$. 
\end{corollary}

\begin{proof}
   This result is simply implied by the inclusion of the global attractor $\mathscr{A}$ in the exponential attractor $\mathscr{E}$,
$$
	\mathscr{A} \subset \mathscr{E}
$$
because $\lim_{t \to \infty} dist_H (S(t) \mathscr{A}, \mathscr{E}) = dist_H (\mathscr{A}, \mathscr{E}) = 0$, and that by definition the exponential attractor $\mathscr{E}$ has a finite fractal dimension. The estimate \eqref{dimA} follows from Theorem \ref{ExpAr}.
\end{proof}	

\begin{corollary} \label{Col2}
      Under the same assumptions as in Theorem \ref{MTh}, the orthogonal projection of the trajectories in the global attractor $\mathscr{A}$ on the finite dimensional subspace $PH$ of the low modes is determining in the sense that, for two trajectories $g(t)$ and $\tg (t)$ in $\mathscr{A}$, if
      $$
      	\|P g(t) - P \tg (t)\|_H \to 0, \quad \text{as} \; \; t \to \infty,
      $$
then
$$
	\|g(t) - \tg (t) \|_H \to 0, \quad \text{as} \; \; t \to \infty.
$$	
Here the finite-rank orthogonal projection $P$ is affiliated with the corresponding squeezing property of the Hindmarsh-Rose semiflow.
\end{corollary}

This Corollary \ref{Col2} is a consequence of the squeezing property of the Hindmarsh-Rose semiflow shown in Theorem \ref{EA} above and Theorem 14.3 in \cite{Rb}.

\vspace{10pt}
\bibliographystyle{amsplain}

\end{document}